\documentclass{amsart}
\usepackage{amsfonts,amssymb, amscd}

\newtheorem{theorem}{Theorem}[section]

\newtheorem{lemma}[theorem]{Lemma}
\newtheorem{corollary}[theorem]{Corollary}
\newtheorem{proposition}[theorem]{Proposition}
\theoremstyle{remark}
\newtheorem{remark}[theorem]{Remark}
\theoremstyle{definition}

\numberwithin{equation}{section}
\makeatother

\newcommand{\I}{1\!{\mathrm l}}

\DeclareMathOperator{\Rdb}{{\mathbb R}}

\DeclareMathOperator{\Ndb}{{\mathbb N}}

\begin{document}

\title[Ueda's theorem]{Ueda's peak set theorem for 
general  
von Neumann algebras}

\date{\today}

\author{David P. Blecher}
\address{Department of Mathematics, University of Houston, Houston, TX
77204-3008}
\email[David P. Blecher]{dblecher@math.uh.edu}
\author{Louis Labuschagne}
\address{DST-NRF CoE in Math. and Stat. Sci,\\ Unit for BMI,\\ Internal Box 209, School of Comp., Stat., $\&$ Math. Sci.\\
NWU, PVT. BAG X6001, 2520 Potchefstroom\\ South Africa}
\email{louis.labuschagne@nwu.ac.za}



\thanks{This work is based on research supported by the National Research Foundation (IPRR Grant 96128), and
the National Science Foundation. Any opinion, findings and conclusions or recommendations expressed in this material, are those of the author, and therefore the NRF do not accept any liability in regard thereto.}

\begin{abstract} 
We extend Ueda's peak set theorem for subdiagonal subalgebras 
of tracial finite von Neumann algebras, to $\sigma$-finite von 
Neumann algebras 
(that is, von Neumann algebras with a faithful state; which includes
those on a separable Hilbert space, or with separable predual.)
To achieve this extension completely new 
strategies had to be invented at certain key points, ultimately 
resulting in a more operator algebraic proof of the result. 
Ueda showed in the case of finite von Neumann algebras 
that his peak set theorem 
is the fountainhead of many other very elegant results, like the 
uniqueness of the predual of such subalgebras, a highly refined 
F \& M Riesz type theorem, and a Gleason-Whitney theorem.    The same is true in our more general 
setting, and indeed we obtain a quite strong variant of the last mentioned theorem.  
We also show that set theoretic issues dash hopes for extending the theorem to  some other large general 
classes of von Neumann algebras,
for example finite or  semi-finite ones.     Indeed certain cases of Ueda's peak set theorem, for a von Neumann algebra $M$, may be seen as  `set theoretic statements' about $M$ that require the 
sets to not be `too large'.  
\end{abstract}

\maketitle

\vspace{.25in}

\section{Introduction}  

\bigskip

In a series of papers the authors extended 
most of the theory of generalized
$H^p$ spaces for function algebras from the 1960s
to the setting of Arveson's (finite maximal) {\em subdiagonal
algebras}.    Most of this is summarized in the survey \cite{BL5}.  
We worked in the setting where the subdiagonal
algebra   $A$ was a unital weak* closed subalgebra
of a von Neumann algebra $M$, and  
where $M$  possesses a faithful
normal tracial state.   Ueda followed this work in \cite{U} by
removing a hypothesis involving a dimensional restriction on $A \cap A^*$ 
in four or five of our theorems, and 
also establishing several other beautiful results such as 
the fact that such an $A$ has a unique predual, all of which followed from 
his very impressive noncommutative (Amar-Lederer) peak set type theorem.  (We will 
say more about peak sets and peak projections later in  this introduction when we 
describe notation and technical background.  Also Section 2 of the paper is devoted to general results about 
peak projections, for example giving some useful characterizations of peak projections
in $C^*$-algebras, von Neumann algebras, and general operator algebras that do not
seem to appear explicitly in the literature.)   
Ueda's peak set result may be phrased as saying that  the support projection in $M^{**}$
of a singular state $\varphi$ on $M$ is dominated by a peak projection $p$ for $A$ (so 
$\varphi(p) = 1$) with $p$ in the `singular part' of $M^{**}$ (that is, $p$ annihilates all 
normal functionals on $M$).

With the theory of subdiagonal subalgebras of von Neumann algebras with a faithful
normal tracial state reaching a level of maturity, several authors turned their attention to the more general $\sigma$-finite von Neumann algebras. Important structural results were obtained by Ji, Ohwada, Saito, Bekjan, and Xu \cite{JOS,jisa,Xu,Ji0,Ji,Bek}.    Recently in \cite{Lis} the second author used Haagerup's reduction theory \cite{HJX} to make several significant advances in generalizing aspects of the earlier theory to the  $\sigma$-finite case, most 
notably the Beurling invariant subspace theory.  The present work flowed out of this, 
being a direct continuation of the line of attack in \cite{Lis}. Here we extend Ueda's 
peak set theorem, and its corollaries, to maximal subdiagonal algebras $A$ in more general 
von Neumann algebras $M$, thereby demonstrating that such algebras too for example have a 
unique predual, admit a highly refined  F \& M. Riesz type theorem, 
have a powerful
 Gleason-Whitney theorem (in particular, 
every normal functional on $A$ has a 
unique Hahn-Banach extension to $M$,  and this extension is also 
normal), etc.  
  We remark that  special cases of 
two of  these results were obtained under an additional semi-finite hypothesis
in \cite{U2}.    
 The technically difficult extension of Ueda's theorem 
to the general $\sigma$-finite case is found in Section 4, while the applications mentioned a few lines back 
are discussed in  Section 5.   In Section 3 we establish some Kaplansky density
type results for operator spaces and subdiagonal algebras which we need.   

In Section 6 we discuss
the extent to which Ueda's theorem might be generalized beyond the $\sigma$-finite case.  
There is some limited good news: our results will have variants valid for any von Neumann algebra 
under an appropriate condition on its center, since it is
known that any  von Neumann algebra is a  direct sum of algebras of the form 
$M_i = R_i \bar{\otimes} B(K_i)$ for $\sigma$-finite  von Neumann algebras  $R_i$.   
The central projections $e_i$ corresponding to this direct sum will
sometimes allow a decomposition of a maximal subdiagonal algebra $A$ of $M$ as a 
direct sum of subalgebras $A_i$ of the $R_i$, and it is easy to see that 
then these are maximal subdiagonal subalgebras of the $\sigma$-finite 
algebras $R_i$.    
The `bad news' is that there is little hope of proving Ueda's theorem 
in ZFC for 
all von Neumann algebras, or even for commutative (and hence finite) or semi-finite von Neumann algebras.
Indeed we show that  the validity
  Ueda's theorem for commutative atomic  von Neumann algebras is a stronger
statement than (it would imply) a ZFC proof of the nonexistence of uncountable
measurable cardinals, a famous 
problem in set theory which nobody today seems to believe is solvable.     
Indeed certain cases of Ueda's peak set theorem, for a von Neumann algebra $M$, may be seen as  `set theoretic statements' about $M$ that require the 
sets to not be `too large'.  These issues are discussed in Section 6, and  this also led to
a sequel paper  with Nik Weaver \cite{BW}.   Some of the ramifications of \cite{BW} are 
described at the end of the present paper, for example that that work indicates that one cannot 
generalize Ueda's peak set theorem  in ZFC much beyond the $\sigma$-finite case (not even 
to $l^\infty(\Rdb)$).   Thus the main result of our paper is somewhat sharp.

We now turn to our set-up, background, and notation.
We recall that a $\sigma$-{\em finite von Neumann algebra} $M$ is one with the property that 
every collection of mutually orthogonal projections is at most countable.
Equivalently, $M$ has a faithful normal state (or even just a faithful state); or has a faithful normal representation possessing a
cyclic separating vector.  We often write 
 $\I$ for the identity of $M$, which may be viewed as the identity
 operator on the underlying Hilbert spaces $M$ is acting on.   
A projection $p \in M$ is called {\em finite}
 if it is not Murray von Neumann equivalent to any proper subprojection;
$M$ is said to be finite if $\I$ is finite.
Beware: $\sigma$-finite von Neumann algebras are not sums of finite ones,
nor is every finite von Neumann algebra  $\sigma$-finite.   However 
 a von Neumann algebra $M$ possesses a faithful
normal tracial state (the setting of \cite{BL5} and most of \cite{AIOA}) if and only if it 
is both finite and  $\sigma$-finite.  (For the difficult direction of this one may compose the center valued trace
on a finite  von Neumann algebra, with a  faithful
normal state on the center, which in this case is $\sigma$-finite.   From this it  follows easily from e.g.\ \cite[Proposition 2.2.5]{Sakai})
that any finite  von Neumann algebra  is a direct sum of algebras with 
 faithful
normal tracial states.)   Any von Neumann algebra which is
separably acting, or
 equivalently has separable predual $M_*$, is $\sigma$-finite.
We will sometimes mention  {\em semi-finite von Neumann algebras}; that is 
$1$ is a sum of mutually orthogonal finite projections, or equivalently 
that every nonzero projection has a  nonzero finite subprojection.  

For a  subalgebra $A$ of $C(K)$, the continuous scalar functions on a compact Hausdorff space $K$, 
a {\em peak set} is a set of form $f^{-1}(\{ 1 \})$ for $f \in A, \Vert f \Vert = 1$.  By replacing
$f$ by $(1+f)/2$ we may assume also that $|f| = 1$ only on $E$.    A noncommutative version 
of this called  {\em peak projections} was considered in \cite{Hay} and developed there
and  in a series of papers
e.g.\  \cite{BHN,Bnpi,BRead,BN,BReadII,BRst}.     There are
various useful equivalent definitions of peak projections in
the latter papers.   They may be defined to be the weak* limits  $q = \lim_n \, a^n$
 in the bidual
for $a \in {\rm Ball}(A)$ in the
case such limit exists \cite[Lemma 1.3]{BReadII}.
We will say much more about peak projections in Section 2 below.

Let $M$ be a $\sigma$-finite von Neumann algebra, and let
$\nu$ be a fixed faithful
normal state on $M$.  We write $N$ for the crossed product $M\rtimes_\nu \mathbb{R}$
of $M$ with the modular group 
$(\sigma_{t}^\nu)$ induced by $\nu$. If $M$ acts on the Hilbert space
$\mathfrak{H}$, this crossed product is constructed by canonically representing the elements $a$
of $M$ as operators on $L^2(\mathbb{R},\mathfrak{H})$ by means of the prescription
$\pi(a)\xi(t)=\sigma_{-t}^\nu(a)(\xi(t))$, and then
generating a ``larger'' von Neumann algebra by means of the elements $\pi(a)$ and the
shift operators $\lambda_s(\xi)(t)=\xi(t-s)$. The crossed product is known to admit a dual
action of $\mathbb{R}$ in the form of an automorphism group $(\theta_s)$
indexed by $\mathbb{R}$, and a normal faithful semifinite trace $\tau$ characterised 
by the property that $\tau\circ \theta_s=e^{-s}\tau$. (See \cite{Terp}.)

The identification $a\to\pi(a)$ above turns out to be a
*-isomorphic embedding of $M$ into $N$, and we will for the sake of simplicity identify
$M$ with $\pi(M)$. For simplicity of notation the canonical Hilbert space on which $N$ acts will be denoted by $K$
rather than $L^2(\mathbb{R},\mathfrak{H})$.
We will work in the space $\tilde{N}$
of all
$\tau$-measurable operators on $K$ affiliated to $N$. We remind the reader that the $\tau$-measurable operators are those
closed densely defined affiliated operators $f$ which are ``almost'' bounded in the sense that for any $\epsilon>0$ we may find a projection $e\in N$ with $\tau(\I-e)<\epsilon$, and with $fe\in N$. This space turns out to be a very well-behaved complete *-algebra large enough to admit all the noncommutative function spaces of interest.  
Using the fact that $\tau\circ \theta_s=e^{-s}\tau$, it is a simple matter to show that the group of *-automorphisms $(\theta_s)$ above admits an extension to continuous $*$-automorphisms on $\tilde{N}$ (see for example either of \cite[bottom p.\ 42]{G-L} and \cite[Proposition 4.7]{LL2}). We will retain the notation $\theta_s$ for these extensions.  
Within this framework, the Haagerup $L^p$-spaces 
($0<p<\infty$) are defined by $L^p(M)=\{a\in \tilde{N}: \theta_s(a)=e^{-s/p}a, s\in \mathbb{R}\}$. 
Having thus being identified as subspaces of $\tilde{N}$, the Haagerup $L^p$-spaces in a very natural way inherit both a conjugation and an order structure from $\tilde{N}$. In fact even the natural (quasi-)norm topology on $L^p(M)$ is inherited from $\tilde{N}$ -- 
the (quasi-)norm topology on each $L^p(M)$ agrees with the subspace topology inherited from $\tilde{N}$ \cite[Proposition II.26]{Terp}.  
  Via the canonical embedding of $L^1(M)$ in $\tilde{N}$ a functional
$\psi \in M_*$ is positive (resp.\ selfadjoint, i.e.\ a difference of two positive functionals) if and only if 
 its image in $\tilde{N}$ is positive (resp.\ selfadjoint) \cite{Terp}.

We remind the reader that the crossed product admits an operator valued weight from the extended positive part of $N$ to that of $M$.  Using this operator valued weight, any normal weight $\omega$ on $M$ may be extended to a dual weight $\widetilde{\omega}$ on $N$ by means of the simple prescription $\widetilde{\omega}=\omega\circ T$. In our analysis $h$ will denote the Radon-Nikodym derivative of the (faithful normal semifinite) dual weight $\tilde{\nu}$ of 
our fixed faithful
normal state $\nu$ above with respect to
$\tau$. It is known that in the $\sigma$-finite case, $h$ belongs to the positive cone of the Haagerup space $L^1(M)$.
Using this operator it is also known that for each $0\leq c\leq 1$, $a\to h^{c/2}ah^{(1-c)/p}$ defines a dense embedding of $M$ into
$L^p(M)$ ($1\leq p<\infty$) \cite{Kos}. Inspired by this fact, the Hardy spaces $H^p(A)$ ($1\leq p<\infty$) are defined to be the closure in $L^p(M)$ of  the subspace $h^{c/p}Ah^{(1-c)/p}$ where $0\leq c\leq 1$. (We remind the reader that the closures for the various values of $c$ all agree \cite{Ji}).

Given such a von Neumann algebra $M$ and ${\mathcal E}$ a  faithful normal conditional expectation from $M$ onto
a von Neumann subalgebra $D$, a {\em subdiagonal algebra} $A$ in $M$ (with respect to ${\mathcal E}$)  is defined to be
a weak* closed subalgebra
of $M$ containing $\I$ such that $A + A^*$ is weak* dense in $M$,
and for which the action of the 
conditional expectation ${\mathcal E} : M \to D = A \cap A^*$ is multiplicative on $A$.
We say that $A$ is {\em maximal} if it is not properly contained in
any larger proper subdiagonal algebra in $M$ with respect to ${\mathcal E}$.
Maximality of such unital weak* closed subdiagonal algebras satisfying the aforementioned weak* density condition, is characterised by the requirement that $A$ be invariant under the
modular automorphism group 
$(\sigma_t^\nu)$ introduced a few paragraphs
earlier ($\nu$ is as above).  
(see \cite[Theorem 1.1]{Lis}, or equivalently \cite[Theorem 1.1]{Xu} \& \cite[Theorem 2.4]{JOS}).

Since we will have occasion to use the 
Haagerup reduction theorem \cite{HJX}, we pause
to explain the essentials of that theory. From the von Neumann algebra $M$ one constructs a larger
algebra $R$ by computing the crossed product with the diadic rationals $\mathbb{Q}_D$ (not $\mathbb{R}$). 
 So in this case one uses only the *-automorphisms $\sigma^\nu_t$ with symbols $t$ in $\mathbb{Q}_D$ to similarly 
 construct a copy $\pi_{\mathbb{Q}_D}(M)$ of $M$ inside $B(L^2(\mathbb{Q}_D,\mathfrak{H}))$, with 
$R = M\rtimes_\nu \mathbb{Q}_D$ then being the algebra generated by the elements belonging to this copy of $M$, 
and the shift operators $\lambda_s$ with symbol $s$ in $\mathbb{Q}_D$. The discreteness of
the group ensures that in this case the associated operator valued weight from the extended positive part
of $R$ to that of $M$, is in fact a faithful normal conditional expectation $\Phi:R\to M$. Inside $R$ one may
then construct an increasing net $R_n$ of finite von Neumann algebras and a concomitant net of 
expectations $\Phi_n:R\to R_n$ for which $\Phi_n\circ \Phi_m=\Phi_m\circ\Phi_n=\Phi_n$ when $n\geq m$. (In the present setting this net actually turns out to be a sequence.) Each $R_n$ comes equipped with a faithful state $\tilde{\nu}_n=\nu\circ\Phi|_{R_n}$ and a faithful normal tracial state $\tau_n$.

The vital fact regarding this construction, is that it may be adapted to the study of maximal subdiagonal algebras.
Following \cite{Xu}, let $\hat{D}$ be the von Neumann subalgebra of $R$ generated by $D$ and the
shift operators $\lambda_s$ ($s\in \mathbb{Q}_D$). (This is in essence just a copy
of $D\rtimes_{\sigma^{\nu}}\mathbb{Q}_D$.) Similarly let
$\hat{A}$ be the weak* closed subalgebra generated by $A$ and the same set of shift operators.   Since
 $A$ is invariant under $\sigma_t^\varphi$ in that reference,
$\hat{A}$ may be defined as the weak* closure of sums of terms of the
form $\lambda(t) x$ with $x \in A$.  
 It is shown
in \cite{Xu} that $\hat{A} \cap \hat{A}^* = \hat{D}$.
The canonical expectation $\mathcal{E}:M\to D$ extends to an expectation
$\hat{{\mathcal E}} : R \to \hat{D}$, and if indeed $A$ is maximal subdiagonal with respect to
$\mathcal{E}$, the algebra $\hat{A}$ will be maximal subdiagonal with respect to $\hat{{\mathcal E}}$.
Moreover the expectation $\Phi:R\to M$, maps $\hat{A}$ and $\hat{D}$
onto  $A$ and $D$ respectively.
By equations (2.5) and (3.2) in \cite{Xu}, and the fact which we mentioned
a  few lines back
 regarding the definition of $\hat{A}$, we see that
  ${\mathcal E} \circ \Phi = \Phi \circ \hat{{\mathcal E}}$ on $\hat{A}$.  Hence 
$\Phi(\hat{A}_0) = A_0$ since
if $\hat{{\mathcal E}}(\hat{a}) = 0$ then
${\mathcal E}(\Phi(\hat{a})) = \Phi(\hat{{\mathcal E}}(\hat{a})) = 0$.

Taking this one step further, the subalgebras $\hat{A}_n=\hat{A}\cap R_n$ turn out to be 
finite maximal subdiagonal subalgebras
of the finite von Neumann algebras $R_n$, with the restriction of $\hat{{\mathcal E}}$
to $R_n$ acting multiplicatively on $\hat{A}_n$, and mapping $R_n$ onto $\hat{D}\cap R_n$.
The algebras $\hat{A}_n$ turn out to be an increasing sequence of algebras which are weak* dense
in $\hat{A}$.

\section{Peak projections}  

As we said in the introduction, peak projections with respect to an operator algebra
$A$ may be defined to be the weak* limits  $q = \lim_n \, a^n$
 in the bidual, 
for $a \in {\rm Ball}(A)$ in the
case such limit exists.      Historically,
if $A$ is a $C^*$-algebra $B$ then peak projections are  very closely related to  Edwards and Ruttiman's element $u(a)$ (see e.g.\ \cite{ER4}),  computed
in $B^{**}$.  Certainly they are the same if $a \geq 0$, and in that case they also agree with
the $B^{**}$-valued Borel functional calculus element $\I_{\{ 1 \}}(a)$.   Also they are the
same, that is $q$ above is $u(a)$, if $\Vert 1 - 2a \Vert \leq 1$ (see \cite[Corollary 3.3]{BN}).    
Thus we shall sometimes simply 
write our peak projections as $u(a)$.   
Indeed  every  peak projection is of the form 
$u(x)$ where $\Vert 1-2x \Vert \leq 1$ (if $A$ is unital and $a^n \to q$ in the weak* topology
 with $a \in {\rm Ball}(A)$
set $x = \frac{1}{2}(1+a)$ \cite[Corollary 6.9]{BHN},
or see \cite[Theorem 3.4 (3)]{BN} for the general case).  
We call $u(x)$ the peak for $x$.     There is an elementary connection 
with the {\em support projection} $s(\cdot)$ (computed in $B^{**}$) which is often useful:  if $B$ is a unital $C^*$-algebra then 
 $$u(1-x) = \I_{\{ 1 \}}(1-x) = \I_{(0,\infty)}(x)^\perp = s(x)^\perp , \qquad x \in {\rm Ball}(B)_+ .$$
A similar but more general result holds in a unital nonselfadjoint algebra $A$: in the notation of Proposition 2.22 in \cite{BRead} 
(see also \cite[Proposition 5.4]{Hay}) that result 
says that if $\Vert 1-2x \Vert \leq 1$ then the peak for $x$ is
$u(x) = s(1-x)^\perp$, where $s(\cdot)$ is the support projection in $A^{**}$ studied in e.g.\ \cite[Section 2]{BRead}.  

The following fact is implicit in the noncommutative peak set theory (see e.g.\ \cite{BRead,Bnpi,BN}), but we could not find it stated explicitly (except in the case of
two projections--see e.g.\ \cite{Hay}).    

\begin{lemma} \label{cipeak}  If $A$ is a closed subalgebra of a $C^*$-algebra $B$
then the infimum of any countable collection of peak projections for $A$ is a peak projection
for $A$. \end{lemma}

\begin{proof}   
  We may assume that $A$ is unital, for example by Proposition 
\cite[Proposition 6.4 (1)]{BReadII} (see also \cite[Lemma 3.1]{BN}).
Suppose that $q_n$ is a peak for $a_n \in A$, and that $\Vert 1 - 2 a_n \Vert \leq 1$ (which can 
always be arranged as we said). 
Let $q = \wedge_n \, q_n$, and $a = \sum_n \, \frac{a_n}{2^n}$.       We
will show that $q$ is the peak for $a$.   
  By a relation above the lemma we have 
 $$u(a) = s(1-a)^\perp = s(\sum_n \, \frac{1-a_n}{2^n})^\perp
= (\vee_n \, s(1-a_n))^\perp = \wedge_n \, s(1-a_n)^\perp = \wedge_n \, q_n.$$
In the last line we have used e.g.\ Proposition 2.14 or Theorem 2.16 (2)  in \cite{BRead}, and the 
easy and known fact that the support projection of the closure
of a sum of right ideals with left contractive approximate identities
is the supremum of the individual support 
projections \cite{BRead}.  
\end{proof}

\begin{remark}  There is also a `facial' proof of the previous result
along the lines of \cite[Proposition 1.1]{BN}.  Another proof follows
from an appeal to the next two results.
 \end{remark}  

For a compact Hausdorff space $K$, the peak sets for $C(K)$ can be characterized abstractly as
the compact  $G_\delta$ subsets.   There is a similar fact for $C^*$-algebras using Akemann's noncommutative 
topology (see \cite{AAP} and references therein): the next result chararacterizes the peak projections
for any $C^*$-algebra $B$ as the `compact $G_\delta$ projections'.   A  $G_\delta$ projection
is the infimum in  $B^{**}$ of a sequence $(p_n)$ of open projections in $B^{**}$, where 
a projection in $B^{**}$ is said to be open if it is  a weak* limit of an increasing net from $B_{+}$.
 The orthogonal complement of an open projection is called closed.
A  compact  projection in $B^{**}$ is 
a  projection  $q \in B^{**}$  which is closed and satisfies $q a = q$ for some $a \in {\rm Ball}(B)_+$
(or equivalently, which is closed with respect to  $B^1$; see e.g.\ \cite{AAP}, or {\rm 2.47} in \cite{Brown1}).  If $B$ is unital then `compact' is the same as `closed'.

We have not been able to find the following result in the literature except for some form
(see e.g.\ \cite{Brown1}) of parts of the unital case:  

 \begin{proposition} \label{cipeak3}  If $B$ is a $C^*$-algebra and $q$ is a projection in $B^{**}$, the following are equivalent: \begin{itemize} 
\item   $q$ is a peak projection with respect to $B$.
\item    $q$ is a compact $G_\delta$ projection.    
\item     $q$ is the weak* limit of a decreasing sequence from $B_{\rm sa}$.  \end{itemize} 
\end{proposition}

\begin{proof}     (1) $\Rightarrow$ (2) \ 
 If $q = u(a)$ for $a \in {\rm Ball}(B)_+$ let $p_n$ be the $B^{**}$-valued spectral projection of $(1-\frac{1}{n},1+\frac{1}{n})$
for $a$.   This gives a  decreasing sequence  of  open 
projections in $B^{**}$ whose infimum (= weak* limit) equals $q$ by the Borel functional calculus.
It is well known that peak projections are compact (e.g.\ since $q = aq$).   

(1) $\Rightarrow$ (3) \ Clearly $a^n \searrow u(a)$ weak* if $a \in {\rm Ball}(B)_+$ .

If $B$ is unital then one may finish the proof using the relation 
$u(1-x) =  s(x)^\perp$ discussed above,  
and known results about the support projection $s(\cdot)$.    Thus (2) implies by e.g.\ \cite[Corollary 3.34]{Brown1}  that $1-q$ is a support projection, so that $q$ is a peak projection.
Similarly if $B$ is unital then (3) implies that $1-q$ is the weak* limit of an increasing sequence $(a_n)$ from $B_+$.   Let $k = \sum_{k=1}^\infty \, \frac{1}{2^n} \, a_n$, then $k \leq 1-q$.
A standard argument shows that $k$ is strictly positive in the hereditary subalgebra (HSA) determined by $1-q$ (any state of that HSA annihilating $k$ 
also annihilates each $a_n$, hence also $1-q$, which is impossible).    Thus $1-q$ is the support projection of $k$, so that $q$ is the peak projection of $1-k$.   

If $B$ is nonunital then  (2) or (3) imply similar conditions with respect to $B^1$, so that by the unital case  $q$ is a peak for $a+t1$ for some $t \in [0,1]$ and 
$a  \in {\rm Ball}(B)_+$.   The norm of $a +t 1$ is $\Vert a \Vert + t = 1$, and so $0 \leq t = 1- \Vert a \Vert < 1$ (or else $a = 0$ which
is impossible).      It is then easy to see, 
by e.g.\ the functional calculus for $a$, that $q = u(a+t1) = u(a/\Vert a \Vert)$, giving (1).
\end{proof} 

We now describe  general  peak projections in terms of the $C^*$-algebraic peak  projections characterized in the 
last result.

\begin{lemma} \label{cipeak2}  If $A$ is a closed subalgebra of a $C^*$-algebra $B$ and $q \in B^{**}$ 
then $q$ is a peak projection for $A$  if and only if $q  \in A^{\perp \perp}$ and $q$ is a peak  projection for $B$. \end{lemma}

\begin{proof}     If $q$ is a peak projection for $A$, the peak for 
$x \in {\rm Ball}(A)$,
then $q$ is the weak* limit of $(x^n)$, 
which  is in $A^{\perp \perp}$.  It is also the peak for some $a \in {\rm Ball}(B)_+$ 
(e.g.\ for $x^* x$ or $|x|$, this follows for example from the proof of \cite[Lemma 3.1]{BN} or from the formula
$u(x^* x) = u(x)^* u(x)$).

The converse is Corollary 4.5 of [12], and has
even been generalized to Jordan operator algebras in a recent paper of the first author with
Neal.  A sketch of a proof: Suppose
 that  $q  \in A^{\perp \perp}$ and $q$ is a peak  projection for $B$.  
We may assume that $B$ is unital.  
Let $A^1$ be the span of $A$ and $1_B$.  By Proposition \ref{cipeak3} (ii) and 
\cite[Lemma 4.4]{BRst}, $q$ is a peak projection for $A^1$.  
\cite[Proposition 6.4 (1)]{BReadII} (see also \cite[Lemma 3.1]{BN}), $q$ is a peak projection for $A$. 
 \end{proof}

The following result, which characterizes peak projections
in subalgebras of von Neumann algebras, will also be used in \cite{BW}.

\begin{theorem} \label{psvn}  If $A$ is a closed subalgebra (not necessarily with 
any kind of approximate identity) of a von Neumann algebra $M$ 
and $q$ is a  projection in $M^{**}$ , then $q$ is a peak projection for $A$
 if and only if $q  \in A^{\perp \perp}$ and $q = \wedge_n \, q_n$, the infimum in  $M^{**}$ of a (decreasing, if desired) sequence $(q_n)$ of 
projections in $M$.  \end{theorem}

\begin{proof}   If $q$ is a peak projection for $x \in {\rm Ball}(A)$ then by 
the last lemma $q$ is in $A^{\perp \perp}$,
and $q$ is the peak for some $a \in {\rm Ball}(M)_+$, so that 
$q =  \I_{\{ 1 \}}(a)$, the $M^{**}$-valued spectral projection of $\{ 1 \}$.   
Let $q_n$ be the $M$-valued spectral projection of $(1-\frac{1}{n},1+\frac{1}{n})$
for $a$.   We claim that the decreasing sequence  $(q_n)$ in $M$ has infimum $q$ in $M^{**}$.
To see this note that as in Proposition \ref{cipeak3},   $q$ is the infimum of $(p_n)$ in $M^{**}$ where $p_n$ is
the $M^{**}$-valued spectral projection of $(1-\frac{1}{n},1+\frac{1}{n})$
for $a$.       However, $q_n \leq p_n \searrow q$.   This may be seen 
from viewing the $M$-valued Borel functional calculus  as the
$M^{**}$-valued Borel functional calculus multiplied by the canonical central projection $z$ with
$z M^{**} \cong M$ (this follows in turn from the uniqueness property of the Borel functional calculus).
Also $q \leq q_n$ (as may be seen e.g.\ by the above functional calculi, using continuous
$f$ with $\I_{\{ 1 \}} \leq f \leq \I_{(1-\frac{1}{n},1+\frac{1}{n})}$).  

Conversely, suppose that $q  = \wedge_n \, q_n$.  Note that
 $q_n$ is clearly a peak projection for $M$,
hence so is $q$ by Lemma \ref{cipeak}.     Now apply the last lemma.
\end{proof}  

\section{A Kaplansky density type result}

The following simple principle
will be useful 
for dealing with Kaplansky density type results in 
unital operator spaces.

\begin{lemma} \label{kos}   Let $M$ be a unital operator space
or operator system.  Let $\sigma$ be any linear topology on $M$ weaker than the
norm topology, e.g. the weak or weak* topology (the latter
if $M$ is a dual space too).  Let $X$ be a subspace of $M$ for which
${\rm Ball}(X)$ is dense in ${\rm Ball}(M)$ in the topology $\sigma$.  Then
$\{ x \in X : x + x^* \geq 0 \}$ is  dense in
$\{ x \in M : x + x^* \geq 0 \}$  in the topology $\sigma$.
\end{lemma} \begin{proof}  Suppose that
$x \in M$ with
$x + x^* \geq 0$.  Then
$z = x + \frac{1}{n}$ satisfies
$z + z^* \geq 0$ and
$$z + z^* \geq \frac{2}{n} \geq
C  z^* z$$
for some constant $C > 0$.  This implies that
$C^2 z^* z - C(z + z^*) + 1 = (1 - Cz)^* (1-Cz) \leq 1$.
We may then approximate $1-Cz$   in the topology $\sigma$
by a net $x_t \in {\rm Ball}(X)$,
 and so $\frac{1}{C} (1-x_t) \to z$  with respect to $\sigma$.
Since $2 - x_t - x_t^* \geq 0$ we have shown that
$z$ is in the closure of $\{ x \in X : x + x^* \geq 0 \}$
 in the topology $\sigma$.
Hence so is $x$.  
  \end{proof}

The following is a Kaplansky density type result generalizing the 
one in Corollary 4.3 in \cite{BL3}, and \cite[Section 4]{U}
(where Ueda points out that the dimensional restriction in 
\cite[Corollary 4.3]{BL3} can be removed).

\begin{theorem} \label{kap1}  If $A$ is a maximal subdiagonal algebra 
in a $\sigma$-finite von Neumann algebra $M$, then 
${\rm Ball}(A + A^*)$ is weak* dense in ${\rm Ball}(M)$.
Hence ${\rm Ball}(A + A^*)_{\rm sa}$ is weak* dense in 
${\rm Ball}(M)_{\rm sa}$.
Moreover, $(A + A^*)_+$ is weak* dense in $M_+$.     Also, in all of these statements we can
replace  `weak*' by $\sigma$-strong*.
\end{theorem}

\begin{proof}  
The first assertion is known in the case that $M$ has a 
faithful normal tracial state (this is the case 
discussed immediately before the theorem).   
Let $x \in {\rm Ball}(M)$.   As stated in the 
introduction, one may construct a $\sigma$-finite von Neumann super-algebra 
$R$ of $M$ with $M$ appearing as the image of a faithful normal 
conditional expectation $\Phi:R\to M$. This $R$ may be constructed so 
that it appears as the weak* closure of an increasing sequence $R_n$ of 
finite von Neumann algebras each of which is the image of a faithful normal 
conditional expectation $\Phi_n:R\to R_n$ for which we have that
$\Phi_n\circ \Phi_m=\Phi_m\circ\Phi_n=\Phi_n$ when $n\geq m$. 
In fact each $x\in R$ is the weak* limit of the sequence 
$\Phi_n(x)$. 

As shown by \cite{Xu}, this construction can be modified in such a 
way that $R$ admits a maximal subdiagonal subalgebra $\hat{A}$ for which 
$\Phi$ will map $\hat{A}$ onto $A$, and $\hat{A}\cap \hat{A}^*$ onto 
$A\cap A^*$. Moreover the subalgebras $\hat{A}_n\cap R_n \subset
 R_n$, are each maximal subdiagonal in $R_n$, with $\cup_{n=1}^\infty\hat{A}_n$ 
weak*-dense in $\hat{A}$. By known results ${\rm Ball}(\hat{A}_n + \hat{A}_n^*)$ 
is for each $n$ weak* dense in ${\rm Ball}(R_n)$. So the subset 
$\cup_{n=1}^\infty {\rm Ball}(\hat{A}_n + \hat{A}_n^*)$ of 
${\rm Ball}(\hat{A} + \hat{A}^*)$ must be weak* dense in the weak* closure of 
$\cup_{n=1}^\infty{\rm Ball}(R_n)$, namely ${\rm Ball}(R)$. It therefore follows that 
$\Phi({\rm Ball}(\hat{A} + \hat{A}^*)) = {\rm Ball}(A + A^*)$
is weak* dense in $\Phi({\rm Ball}(R)) = {\rm Ball}(M)$.

The second assertion follows from the first by taking the 
real part.  The third follows by applying the previous Lemma to the first.   
The last assertion follows from the previous assertions
 and \cite[Theorem 2.6 (iv)]{Tak}.
\end{proof}

We give a corollary of this which we will use later.
Note that any element in $A + A^*$ has a unique 
representation $a^* + d + b$ with $a, b \in A_0$ and $d \in D$.
This is because if $a^* + d + b = 0$ then applying ${\mathcal E}$
we see that $d = 0$.  Also $A_0 \cap A_0^* \subset
 D \cap A_0 = (0)$.  Thus $A + A^* = A_0 \oplus D \oplus A_0^*$.
It follows from this that selfadjoint elements $x$ in $A+A^*$ are of form $a + d + a^*$ for $a \in A_0, d \in D_{\rm sa}$; and $d$ 
must be positive if $x \geq 0$ since $d = {\mathcal E}(x)$.
We write $H$ for the Hilbert transform on $L^2(M)$ with respect to $A$
as presented in {\rm \cite{Ji}}.  It is shown there that $H$
 is continuous on $L^2(M)$. 
In this insightful paper Ji shows that for each fixed $1<p<\infty$ the operators 
$$H_\theta(h^{\theta/p}(a+d+b^*)h^{(1-\theta)/p} = ih^{\theta/p}(b^* -a)h^{(1-\theta)/p}, \qquad a, b \in A, d \in D, \theta\in[0,1],$$ extend to a unique bounded operator on $L^p(M)$ (the Hilbert transform) not dependent on the parameter $\theta$. See \cite[Theorem 3.2]{Ji}.
In the case $p=\infty$, the Hilbert transform $H$ is only partially defined on $M$ by means of the formula by $H(a+d+b^*) = i(b^* -a)$, for $a, b \in A, d \in D$.   
We remind the reader that $h$ is the Radon-Nikodym derivative of the (faithful normal semifinite) dual weight $\tilde{\nu}$ with respect to the canonical trace on $N$, and its role in the definition of $H^p(A)$ is
discussed in the introduction. 
  Similarly, the selfadjointness and positivity referred to below was discussed there 
too.

\begin{lemma} \label{maxprx}  Let $A$ be as in the previous result,
and $H$ the Hilbert transform on $L^2(M)$ with respect to 
$A$. If $x \in M_{sa}$, then $h^{\frac{1}{2}} \, H(x h^{\frac{1}{2}})$ is
selfadjoint. Moreover $H(x h^{\frac{1}{2}})^*=H(h^{\frac{1}{2}}x)$.
 \end{lemma}

\begin{proof} 
It suffices to prove the claim for the case where $x\in M_+$. 
If $a \in A_0, d \in D_{\rm sa}$ then $H((a^* + d + a) h^{\frac{1}{2}}) 
= i (a^* - a) h^{\frac{1}{2}}$ by the definition
in \cite{Ji}.  Hence $h^{\frac{1}{2}} \, H((a^* + d + a) h^{\frac{1}{2}})
= i h^{\frac{1}{2}} (a^* - a) h^{\frac{1}{2}}$ is selfadjoint.
Any $x \in M_+$ is the weak* limit 
of a net $x_\lambda = a_\lambda^* + d_\lambda + a_\lambda$, where $a_\lambda\in A$ and $d_\lambda\in D_+$, 
by Theorem \ref{kap1} and the comment following it.
Hence the net $(x_\lambda h^{\frac{1}{2}})$ converges weakly to $xh^{\frac{1}{2}}$ in $L^2$. To see this note that for any 
$b\in L^2$, $h^{\tfrac{1}{2}}b$ will be in $L^1$, whence $tr(x_\lambda h^{\frac{1}{2}}b)\to tr(xh^{\frac{1}{2}}b)$.  Since any norm continuous operator is also weakly continuous, the $L^2$ continuity of $H$ ensures that $(H(x_\lambda h^{\frac{1}{2}}))$ converges 
weakly to $H(xh^{\frac{1}{2}})$ in $L^2$. This in turn ensures that $(h^{\frac{1}{2}}\,H(x_\lambda h^{\frac{1}{2}}))$ converges weakly to $h^{\frac{1}{2}}\,H(xh^{\frac{1}{2}})$ in $L^1$.  By the lines at the start 
of this paragraph, $h^{\frac{1}{2}} \, H(x_\lambda h^{\frac{1}{2}})$ is selfadjoint.
Hence $h^{\frac{1}{2}} \, H(x h^{\frac{1}{2}})$ is selfadjoint with respect to the conjugation structure inherited from $\tilde{N}$.
Similarly in view of the fact that $(x_\lambda h^{\frac{1}{2}})^*= h^{\frac{1}{2}}x_\lambda$ is weakly convergent to $(xh^{\frac{1}{2}})^*=h^{\frac{1}{2}}x$, we again have that $H(h^{\frac{1}{2}}x_\lambda)$ is weakly convergent to $H(h^{\frac{1}{2}}x)$.  It is obvious that $H(h^{\frac{1}{2}}x_\lambda)^*=H(x_\lambda h^{\frac{1}{2}})$ for each $\lambda$, from which it follows that $H(x h^{\frac{1}{2}})^*=H(h^{\frac{1}{2}}x)$, as required.
\end{proof}

\section{Ueda's peak set theorem for $\sigma$-finite $M$}  

\begin{theorem} \label{uadasig}   Let $A$ be a maximal subdiagonal subalgebra of a $\sigma$-finite von Neumann algebra  $M$.
For a nonzero singular $\varphi \in M^*$, there exists a contraction $a \in A$ and a projection 
$p \in M^{**}$ with 
\begin{enumerate} \item   $a^n \to p$ weak* in $M^{**}$.
\item 
 $a^n \to 0$ weak* in $M$, or equivalently $\psi(p) = 0$ for all $\psi \in M_*$.
\item
 $|\varphi|(p) = |\varphi|(1)$, where $|\varphi|$ is the absolute value of $\varphi$ regarded 
as a member of the predual of the $W^*$-algebra $M^{**}$.
\end{enumerate}    
\end{theorem}

Since $\varphi$ is known to be singular if and only if $|\varphi|$ is singular \cite{Tak},
one may assume that $\varphi$ is a state if one wishes.  In this
case as in \cite{U}, (1)--(3) may be restated as saying that 
(1) \ $p$ is a peak projection, (2) \ 
$p$ is dominated by the `singular part' projection of $M^{**}$,
and (3) \ $\varphi(p) = 1$.

The present section is devoted to generalizing Ueda's 
elegant proof of the tracial state case of Theorem \ref{uadasig}.
As in Theorem 1 of \cite{U}  we may find a decreasing sequence $(p_n)$ 
of projections in $M$ with strong limit $0$ and $|\varphi|(p_n) =
|\varphi|(p_0) =  
|\varphi|(1) \neq 0$ for all $n$, where $p_0$ is the strong limit of
$(p_n)$ in $M^{**}$.   We may also assume that 
$\nu(p_n) < n^{-6}$ where $\nu$ is the fixed faithful 
normal state on $M$. 
The formal series $g = \sum_k \, k p_k$, may then be shown to correspond to a well-defined 
positive unbounded closed and densely defined operator affiliated to $M$. Moreover the 
formal prescription $g\to h^{1/2} g$ yields a well-defined embedding of this operator 
into $L^2(M)$. These facts are proved in the following lemma:

\begin{lemma} \label{2lim}
Let the projections $p_n$ be as in the previous paragraphs,  for $n\in \mathbb{N}$.  Then the formal operator $g = \sum_k \, k p_k$ corresponds to a densely defined 
positive selfadjoint operator affiliated to $M$. Moreover $h^{1/2}g$ is a densely defined closable operator for which the closure 
is a well-defined element of $L^2(M)$ which appears as the $L^2$-norm limit of the sequence $(h^{1/2}g_n)\subset L^2(M)$, where 
$g_n=\sum_{k=1}^nkp_k$. (In other words $[h^{1/2}(\sum_{k=1}^\infty \, k p_k)] = \sum_{k=1}^\infty \, k [h^{1/2}p_k]$.)  
Similarly the formal operator $gh^{1/2}$ may be regarded as the sum in $L^2(M)$ of the series 
$\lim_{n\to\infty}g_nh^{1/2} = \sum_{k=1}^\infty \, k [p_kh^{1/2}]$.
\end{lemma}

\begin{proof}  We first prove the claim regarding the affiliation of $g$. For this we will make use of the well known theory of the extended positive part of a von Neumann algebra. (See \cite[\S IX.4]{Tak2}.) Observe that 
$g_n = \sum_{k=1}^n \, k p_k$ may be in a canonical way be regarded as an increasing sequence of  elements of the extended positive part 
of $M$. It is  clear from the discussions following \cite[Definition IX.4.4 \& IX.4.6]{Tak2} that the 
supremum of this sequence (which we identify with $g$) is a well-defined element of the extended positive part of $M$. 

Next recall that by \cite[Corollary IX.4.9]{Tak2}, the action of the canonical faithful normal 
state $\nu$ extends to the extended positive part of $M$. 
In terms of this action, we must have that $\nu(g)=\sup_{n\in \mathbb{N}}\nu(g_n)$, and hence that 
$$\nu(g)=\sup_{n\in \mathbb{N}}\nu(g_n) \leq \lim_{n\to \infty}\sum_{k=1}^n k^{-5}        <\infty.$$
However by \cite[Theorem IX.4.8]{Tak2}, $g$ has a spectral decomposition of the form 
$$g(\omega) = \int_0^\infty \lambda\, d\omega(e_\lambda)+\infty\omega(p)\quad \omega\in M_*^+.$$
On considering the case where $\omega=\nu$ and comparing the resulting formula to the the fact that 
$\nu(g)<\infty$, it is clear that we must then have that $\nu(p)=0$, i.e.\ $p=0$. It is now 
not difficult to conclude from the discussion in \cite{Tak2}  that this can only be the case if the 
``operator part'' of $g$ (see \cite[Lemma IX.4.7]{Tak2}) is all of $g$.  See e.g.\ the last paragraph of the proof of Theorem
IX.4.8 there.   Hence $g$ is a densely defined affiliated operator. 

We proceed to verify the claim regarding $h^{1/2}g$  (the $gh^{1/2}$ statement following e.g.\
by duality). Firstly note that by the choice of the 
$p_k$'s we have that $\sum_k \Vert k p_k h^{1/2} \Vert_2 < +\infty$.   Indeed 
$$\| k p_kh^{1/2}\|_2^2=k^2tr(h^{1/2}p_kh^{1/2})=k^2\nu(p_k)<\frac{1}{k^4}.$$
So the formal series $\sum_{k=1}^\infty k[p_k h^{1/2}] = 
\lim_{n\to\infty}g_nh^{1/2}$ must correspond to a well defined $\tau$-measurable element $G$ of  
$L^2(M)\subset \tilde{N}$. 

Recall that $p_mp_k=p_kp_m=p_m$ whenever $m\geq k$. This ensures that 
for any fixed $m\geq 1$ and any $1\leq n \leq \infty$, we have that 
$$g_n(p_{m}-p_{m+1})= \sum_{k=1}^n \, k p_k(p_m-p_{m+1}) = 
(\sum_{k=1}^m \, k)(p_{m}-p_{m+1}).$$(Here $g_\infty$ is identified with $g$.) 
So for each $m\geq 1$ then have that $h^{1/2}g(p_m-p_{m+1}) = 
\lim_{n\to\infty}h^{1/2}g_n(p_m-p_{m+1})$. Taking into account that 
$p_1=\oplus_{m=1}^\infty(p_m-p_{m+1})$, it follows that $h^{1/2}g= h^{1/2}gp_1=\lim_{n\to\infty}h^{1/2}g_np_1=\lim_{n\to\infty}h^{1/2}g_n=G^*$ as required.  
\end{proof}

\begin{remark}   We note that if $g = \sum_n \, n p_n$,
viewed as a supremum in the extended positive part $\hat{M}_+$
of $M$,
then $h^{\frac{1}{2}} g h^{\frac{1}{2}} \in L^1(M)$ and the 
latter can be shown to be the supremum and limit in $L^1(M)$
of $(h^{\frac{1}{2}} g_n h^{\frac{1}{2}})$.  We will not use
this though.  
\end{remark} 

Let $\tilde{g}$ (resp.\ $\tilde{g_n}$)
be the Hilbert transform of $g h^{\frac{1}{2}}$  (resp.\ $g_n h^{\frac{1}{2}}$)
as in  \cite[Section 3]{Ji}, and let $f = g h^{\frac{1}{2}} + i \tilde{g}$
 (resp.\ $f_n = g_n h^{\frac{1}{2}} + i \tilde{g}_n$).
Then $f_n, f \in H^2(A)$.    

In the following result we use the notion of accretive operators (see e.g.\
\cite[Appendix C.7]{Haase}).   In  $L^p(M)$  an operator is accretive if the associated
operator $T \in \tilde{N}$ has $T + T^*$ positive in $\tilde{N}$.

\begin{corollary} \label{deduc}  With
$g = \sum_k \, k p_k$ as above, and $f = g h^{\frac{1}{2}} + i \tilde{g}$, we have 
$h^{\frac{1}{2}} \tilde{g}$ is selfadjoint in $L^1(M)$, so that $h^{\frac{1}{2}} f
= h^{\frac{1}{2}} g h^{\frac{1}{2}} + i h^{\frac{1}{2}} \tilde{g}$ 
is accretive in the sense above.
 \end{corollary}

\begin{proof}
If $g_n$ is as defined above, then 
by Lemma \ref{maxprx} we have $h^{\frac{1}{2}} \, H(g_n
h^{\frac{1}{2}})$ is
selfadjoint.   By Lemma \ref{2lim} 
and the continuity of $H$ from \cite{Ji}, 
it follows that $h^{\frac{1}{2}} \tilde{g}$ is selfadjoint.  Thus, and since $g$ is positive, 
$h^{\frac{1}{2}} f = h^{\frac{1}{2}} g h^{\frac{1}{2}} + 
i h^{\frac{1}{2}} \tilde{g}$ is 
accretive.  
\end{proof}

A  $\sigma$-finite von Neumann algebra $M$ has a convenient `standard form'.  Indeed as we recalled in the introduction,
a characterization
of $\sigma$-finite 
algebras is the existence of
a  (normal faithful) Hilbert space representation $\mathfrak{H}$ possessing a fixed cyclic and separating vector $\Omega$.  Then $\nu(x) = (\Omega, x \ \Omega)$ is a faithful normal state on $M$. It is known that in this context
\begin{equation}
(M, \mathfrak{H}, \mathcal{P}, J, \Omega),
\end{equation} 
is a  `standard form' for $M$, where $\mathcal{P}$ and $J$ respectively denote the naturally associated cone and
the modular conjugation. The modular automorphism group $\sigma_t$ is 
implemented by  $\sigma_t(\cdot) = \Delta^{it}\cdot \Delta^{-it}$, where $\Delta$ is the modular operator. By the universality of the standard form (see \cite{Araki,Haage,Terp}) and hence also of the natural cone, we may identify the context 
$$(M, \mathfrak{H},  \mathcal{P},  J, \Omega)$$
with $$(M,\ L^2(M),\ L^2_+(\mathfrak{M}),\ ^*, h^{\frac{1}{2}}).$$
In what follows we choose to work with the copy of $M$ living inside $B(L^2(M))$ as multiplication operators. In view of the above correspondence, we may do so without loss of generality.   
We view $h^{\frac{1}{2}}$ as the fixed cyclic and separating
vector  for this action of $M$.

\begin{lemma}\label{lem1} For each $k \in \Ndb,$ there exist nets $(a(k)_\lambda)\subset A_0$, $(d(k)_\lambda)\subset D_+$ such that $(a(k)_\lambda^*+d(k)_\lambda+a(k)_\lambda)\in M_+$, with $(a(k)_\lambda^*+d(k)_\lambda+a(k)_\lambda)$ converging to $g_k$ in the $\sigma$-strong* topology. Hence for any $q\in L^2(M)$, the nets $(a(k)_\lambda^*+d(k)_\lambda+a(k)_\lambda)q$ and $q(a(k)_\lambda^*+d(k)_\lambda+a(k)_\lambda)$ will respectively converge in $L^2$-norm to $gq$ and $qg$. 
\end{lemma}

\begin{proof}  
This follows from Theorem \ref{kap1} and the observation following it. 
\end{proof}

The Hilbert transform $H$ in the next result is the map partially defined on $M$
by $H(a+d+b^*) = i(b^* -a)$, for $a, b \in A, d \in D$.   

\begin{lemma}\label{lem2} Given $a\in A_0$, $d\in D_+$ with $a^*+d+a \in M_+$, the element $(a^*+d+a+\I)+iH(a^*+d+a)$ has an inverse $v$ belonging to $A$, with both $v$ and $1-v$ contractive. 
\end{lemma}

\begin{proof} 
Observe that with $a$ and $d$ as in the hypothesis, $H(a^*+d+a)=i(a^*-a)$ is selfadjoint. 
Thus  $x = a^*+d+a + i H(a^*+d+a)$ is accretive.   By the basic theory of accretive operators  (see e.g.\
\cite[Appendix C.7]{Haase}) 
$\I + x$ has a contractive inverse $v$.   Note that $v \in A$ since the numerical range and hence the 
spectrum of $x$ in $A$  is in the right half plane. Also $x (\I + x)^{-1} = \I - (\I + x)^{-1} = \I- v$ 
is  a contraction in $A$, being the average of $\I$ and the well known Cayley transform of $x$.     We remark that
the map $x \mapsto x (\I + x)^{-1}$ 
is called the ${\mathfrak F}$-transform in recent papers of Charles Read and the first author.  
\end{proof}

\begin{proposition}\label{maintech}
There exist elements $(w_k)$ and $w_g$ of $A$ for which each of $w_k$, $w_g$, $w_k-\I$ and $w_g-\I$ are contractions, with
$$w_k[(g_k+\I)h^{1/2}+iH(g_kh^{1/2})]=h^{1/2}=[h^{1/2}(g_k+\I)+iH(h^{1/2}g_k)]w_k$$
and
$$w_g[(g+\I)h^{1/2}+iH(gh^{1/2})]=h^{1/2}=[h^{1/2}(g+\I)+iH(h^{1/2}g)]w_g.$$
Moreover there exists a 
subnet of $(w_k)$ which is weak* convergent to $w_g$.
\end{proposition}

\begin{proof} Choose
 nets $(a(k)_\lambda)\subset A_0$, $(d(k)_\lambda)\subset D_+$
 as in Lemma \ref{lem1}. By Lemma \ref{lem2} each $(a(k)_\lambda^*+d(k)_\lambda+a(k)_\lambda+\I)+iH(a(k)_\lambda^*+d(k)_\lambda+a(k)_\lambda)$ has an inverse $v(k)_\lambda$ belonging to $A$, with both $v(k)_\lambda$ and $1-v(k)_\lambda$ contractive. By passing to a subnet if necessary, we may assume that $(v(k)_\lambda)$ is weak* convergent. Let $w_k$ be the weak* limit of $(v(k)_\lambda)$. (Since both $(v(k)_\lambda)$ and $(v(k)_\lambda-\I)$ are contained in  the weak* compact set Ball$(A)$, it is clear that both $w_k$ and $w_k-\I$ will also be in this set.) We wish to prove that $$w_k[(g_k+\I)h^{1/2}+iH(g_kh^{1/2})]=h^{1/2}=[h^{1/2}(g_k+\I)+iH(h^{1/2}g_k)]w_k.$$ 
In view of the similarity of the proofs, we prove only the first equality. Notice that $(v(k)_\lambda[(g_k+\I)h^{1/2}+iH(g_kh^{1/2})])$ converges weakly in $L^2$ to $w_k[(g_k +\I)h^{1/2}+iH(gh^{1/2})]$.  By Lemma \ref{lem1}
we have that $(a(k)_\lambda^*+d(k)_\lambda+a(k)_\lambda)
h^{1/2} \to g_kh^{1/2}$ in $L^2$-norm, and so  also 
$H((a(k)_\lambda^*+d(k)_\lambda+a(k)_\lambda) h^{1/2}) \to H(g_kh^{1/2})$  
in $L^2$-norm by the continuity of $H$ established in \cite{Ji}.
Since the $v(k)_\lambda$'s are contractive, it easily follows that  
\begin{eqnarray*}
&&\|v(k)_\lambda[(g_k+\I)h^{1/2}+iH(g_kh^{1/2})]-h^{1/2}\|_2\\
&&=\|v(k)_\lambda[(g_k+\I)h^{1/2}+iH(g_kh^{1/2})]\\
&&\qquad\quad- v(k)_\lambda[((a(k)_\lambda^*+d(k)_\lambda+a(k)_\lambda+\I)+iH(a(k)_\lambda^*+d(k)_\lambda+a(k)_\lambda))h^{1/2}]\|_2\\
&&\leq \|[(g_k+\I)h^{1/2}+iH(g_kh^{1/2})]\\
&&\qquad\quad- ((a(k)_\lambda^*+d(k)_\lambda+a(k)_\lambda+\I)+iH(a(k)_\lambda^*+d(k)_\lambda+a(k)_\lambda))h^{1/2}\|_2.
\end{eqnarray*}Hence $(v(k)_\lambda[(g_k+\I)h^{1/2}+iH(g_kh^{1/2})])$ is norm convergent to $h^{1/2}$. The claim regarding the 
$g_k$'s  now follows from the uniqueness of limits.

Since $(w_k)$ is bounded, it will admit a weak* convergent 
subnet $(w_\gamma)$. Let $w_g$ be the limit of that 
subnet. The claim regarding $w_g$ can now be verified with an essentially similar proof, but with the roles of $v(k)_\lambda$ and 
$(a(k)_\lambda^*+d(k)_\lambda+a(k)_\lambda)$ respectively being played by $w_\gamma$ and $g_\gamma$, and with Lemma \ref{2lim} replacing Lemma \ref{lem1}.
Thus we begin by noting that 
$$w_\gamma [(g+\I)h^{1/2}+iH(g h^{1/2})] \; \to \; w_g [(g+\I)h^{1/2}+iH(g h^{1/2})]$$ 
weakly in $L^2$.  Amending the previous argument as described above, now leads to the conclusion that
$$\| w_\gamma [(g+\I)h^{1/2}+iH(g h^{1/2})] - h^{1/2} \|_2 \to 0.$$
So again the claim regarding the $g$'s  follows from the uniqueness of limits.
\end{proof}

We proceed to use Proposition \ref{maintech} to analyse the structure of $[(g+\I)h^{1/2}+iH(gh^{1/2})]$.

\begin{theorem} \label{thobs}  
For any $n\in \mathbb{N}$, we have $\|p_nw_g\|\leq\sqrt{\frac{2}{n(n+1)}}$.
\end{theorem}

\begin{proof} Let $g_k$ and $w_k$ be as in Proposition \ref{maintech}; 
we had there a weak* convergent 
subnet of the latter sequence with limit $w_g$.   As usual one may replace $(w_k)$  by the 
subnet.  For ease of notation, we will assume that $(w_k)$ is weak* convergent to $w_g$. 
It then suffices to show that $\|p_nw_k\|\leq\sqrt{\frac{2}{n(n+1)}}$ for every $k\geq n$. To see this recall that the closed ball of radius $\sqrt{\frac{2}{n(n+1)}}$ is weak* closed. So if each $p_nw_k$ ($k\geq n$) is in this ball, so is $p_nw_g$.

Observe that for $a$, $d$ and $v$ as in Lemma \ref{lem2}, we have 
\begin{eqnarray*}v^*(a+d+a^*+\I)v&=&\frac{1}{2}v^*[((a^*+d+a+\I)+iH(a^*+d+a))\\
&& \qquad\qquad+ ((a^*+d+a+\I)-iH(a^*+d+a))]v\\
&=&\frac{1}{2}[v+v^*].
\end{eqnarray*}
Since $v$ and $v^*$ are both contractive, this means that $v^*(a+d+a^*+\I)v\leq \I$. This in turn ensures that 
\begin{eqnarray*}
&&(a+d+a^*+\I)vv^*(a+d+a^*+\I)\\
&&=(v^{-1})^*v^*(a+d+a^*+\I)vv^*(a+d+a^*+\I)vv^{-1}\\
&&\leq (v^{-1})^*v^*(a+d+a^*+\I)vv^{-1}\\
&&=(a+d+a^*+\I).
\end{eqnarray*} 
Hence 
\begin{eqnarray}\label{ineq}
&& \|v_\lambda^*(a_\lambda^*+d_\lambda+a_\lambda+\I)w_kh^{1/2}a\|^2\\
&&=\langle w_k^* (a_\lambda^*+d_\lambda+a_\lambda +\I)v_\lambda v_\lambda^*(a_\lambda^*+d_\lambda+a)_\lambda+\I)w_kh^{1/2}a\, ,\, h^{1/2}a\rangle\nonumber\\
&&\leq\langle w_k^*(a_\lambda^*+d_\lambda+a_\lambda+\I)w_kh^{1/2}a\, ,\, h^{1/2}a\rangle\nonumber\\
&&=\langle (a_\lambda^*+d_\lambda+a_\lambda+\I)w_kh^{1/2}a\, ,\, w_kh^{1/2}a\rangle\nonumber
\end{eqnarray}

Now let $a\in M$ be given, and let the nets $(a(k)_\lambda)\subset A_0$, $(d(k)_\lambda)\subset D_+$ be as in Lemma \ref{lem1}. Then the nets $(a(k)_\lambda^*+d(k)_\lambda+a(k)_\lambda+\I)w_kh^{1/2}a$ converge to $(g_k+\I)w_kh^{1/2}a$ in $L^2$-norm. As we saw in the proof of Proposition \ref{maintech}, on passing to a subnet if necessary, we may assume that the nets $(v(k)_\lambda)$'s described by Lemma \ref{lem2}, are weak* convergent to the $w_k$'s. 

Since the $v(k)_\lambda$'s are contractive, we have that $$\|[v(k)_\lambda^*(a(k)_\lambda^*+d(k)_\lambda+a(k)_\lambda+\I)w_kh^{1/2}a]-[v(k)_\lambda^*(g_k+\I)w_kh^{1/2}a]\|$$ $$\leq \|[(a(k)_\lambda^*+d(k)_\lambda+a(k)_\lambda+\I)w_kh^{1/2}a]-[(g_k+\I)w_kh^{1/2}a]\| \to 0.$$ 
Thus
$$[v(k)_\lambda^*(a(k)_\lambda^*+d(k)_\lambda+a(k)_\lambda+\I)w_kh^{1/2}a]-[v(k)_\lambda^*(g_k+\I)w_kh^{1/2}a] \to 0$$
in norm.  Since also $v(k)_\lambda^*(g_k+\I)w_kh^{1/2}a$ is weakly convergent in $L^2(M)$ to $w_k^*(g_k+\I)w_kh^{1/2}a$, 
it follows that $v(k)_\lambda^*(a(k)_\lambda^*+d(k)_\lambda+a(k)_\lambda+\I)w_kh^{1/2}a$ is weakly convergent to $w_k^*(g_k+\I)w_kh^{1/2}a$.
 
We proceed to show that $\|(g_k+\I)^{1/2}w_kh^{1/2}a\|_2\leq\|h^{1/2}a\|_2$. To see this we firstly observe that
$$\langle v(k)_\lambda^*(a(k)_\lambda^*+d(k)_\lambda+a(k)_\lambda+\I)w_kh^{1/2}a, h^{1/2}a\rangle$$ $$\to \langle w_k^*(g_k+\I)w_kh^{1/2}a,h^{1/2}a\rangle =\|(g_k+\I)^{1/2}w_kh^{1/2}a\|^2,$$and that 
$$\langle (a(k)_\lambda^*+d(k)_\lambda+a(k)_\lambda+\I)w_kh^{1/2}a\, ,\, w_kh^{1/2}a\rangle$$ $$\to \langle (g_k+\I)w_kh^{1/2}a,w_kh^{1/2}a\rangle= \|(g_k+\I)^{1/2}w_kh^{1/2}a\|^2.$$ 
Next observe that by inequality  
(\ref{ineq}), we have that 
\begin{eqnarray*}
&& \|v(k)_\lambda^*(a(k)_\lambda^*+d(k)_\lambda+a(k)_\lambda+\I)w_kh^{1/2}a\|^2\\
&&=\langle (a(k)_\lambda^*+d(k)_\lambda+a(k)_\lambda+\I)w_kh^{1/2}a\, ,\, w_kh^{1/2}a\rangle
\end{eqnarray*}
It follows from the above inequality that 
\begin{eqnarray*}
&&\langle v(k)_\lambda^*(a(k)_\lambda^*+d(k)_\lambda+a(k)_\lambda+\I)w_kh^{1/2}a\, ,\, h^{1/2}a\rangle\\
&&\leq \|v(k)_\lambda^*(a(k)_\lambda^*+d(k)_\lambda+a(k)_\lambda+\I)w_kh^{1/2}a\|.\|h^{1/2}a\|\\
&&\leq [\langle (a(k)_\lambda^*+d(k)_\lambda+a(k)_\lambda+\I)w_kh^{1/2}a\, ,\, w_kh^{1/2}a\rangle]^{1/2}.\|h^{1/2}a\| .
 \end{eqnarray*}
On taking limits, we have $\|(g_k+\I)^{1/2}w_kh^{1/2}a\|^2\leq \|(g_k+\I)^{1/2}w_kh^{1/2}a\|.\|h^{1/2}a\|,$ or equivalently, $\|(g_k+\I)^{1/2}w_kh^{1/2}a\|\leq \|h^{1/2}a\|$ as claimed.

Finally note that since the $p_n$'s are decreasing, we as before have that $$\frac{(n+1)n}{2}p_n =\sum_{m=1}^n mp_n \leq \sum_{m=1}^n mp_m,$$ which is dominated by $\sum_{m=1}^{k}mp_m=g_m$. Hence 

\begin{eqnarray*}
\|p_nw_k h^{1/2}a\|^2 &=& \langle w_k^* p_n w_k (h^{\frac{1}{2}} a) , (h^{\frac{1}{2}} a) \rangle\\
&\leq& \frac{2}{n(n+1)} \langle w_k^* g_k w_k (h^{\frac{1}{2}}a) , (h^{\frac{1}{2}}  a) \rangle\\
&\leq& \frac{2}{n(n+1)} \langle w_k^* (g_k+\I) w_k (h^{\frac{1}{2}} a) , (h^{\frac{1}{2}}  a) \rangle\\
&=& \frac{2}{n(n+1)} \|(g_k+\I)^{1/2}w_kh^{1/2}a\|^2 \\
&\leq& \frac{2}{n(n+1)} \|h^{1/2}a\|^2.
\end{eqnarray*}
Since the subspace $\{h^{1/2}a: a\in M\}$ is dense in $L^2(M)$, it follows that the operator of left multiplication by $p_nw_k$ on $L^2(M)$, has norm dominated by $\sqrt{\frac{2}{n(n+1)}}$. This proves the claim.
\end{proof}

Thus with $b = 1 - w_g$  we deduce that
$\Vert p_k - p_k b \Vert \leq \sqrt{\frac{2}{k(k+1)}}$
as needed for the argument in \cite{U} to proceed.  Indeed
the rest of that argument is identical. 
 We obtain $p_0 = p_0 b = b p_0$ where 
$p_0$ is the strong limit of
$(p_n)$ in $M^{**}$, and if $a = (1 + b)/2$  then $(a^n)$ converges weak* to a peak projection
$p \geq p_0$ with $|\varphi|(p) = |\varphi|(p_0) = |\varphi|(1)$.  If $\Vert a \xi \Vert_2 = \Vert  \xi \Vert_2$ for  
$\xi \in L^2(M)$, then as in \cite{U} we obtain $b \xi = \xi$, so that in our  notation above 
we have $\xi \in {\rm Ker}(w_g) = 0$.  However Ker$(w_g) = (0)$.   
Indeed the projection associated with the kernel is in $M$; and  if $e \in M$ is a projection with 
$w_g e = 0$ 
then by the last equality in Proposition \ref{maintech}
we obtain $h^{1/2} e = 0$, so that $e= 0$.    Hence as in \cite{U} (which relies 
here on 
the noncommutative peak  theory \cite{Hay}, see also e.g.\
\cite{BHN,BN}) we obtain $a^n \to 0$ weak*
in $M$.  This completes the proof
of the generalization of Ueda's peak set theorem to $\sigma$-finite algebras.

\section{Consequences of Ueda's peak set theorem for $\sigma$-finite $M$}  \label{cons}

All the other consequences from \cite{U} of Ueda's peak set theorem, now go through with unaltered proofs for maximal subdiagonal subalgebras $A$ of a $\sigma$-finite von Neumann algebra $M$.     Indeed this is true rather generally.
If $A$ is a weak* closed subalgebra of a von Neumann algebra $M$ then we say that
$A$ is an {\em Ueda algebra} in $M$
 if Ueda's peak set theorem holds for $A$;
that is if for every singular state on $M$
there is a peak projection $q$ for $A$ with $\varphi(q) = 1$ and
$q$ is dominated by the `singular part' projection of $M^{**}$, as
in the restatement after Theorem \ref{uadasig}.   
The ideas in \cite[Lemma 9.1]{BW} give the following restatement:

\begin{corollary} \label{restU}   Suppose that $A$ is a weak* closed
subalgebra  of a von Neumann algebra $M$.
Then $A$ is an Ueda algebra in $M$  if and only if 
for every  singular state $\varphi$ on $M$, there exists a
(increasing, if desired) sequence $(q_n)$ of projections
in ${\rm Ker}(\varphi)$ with supremum $1$ in $M$,
and supremum in $M^{**}$ lying in $A^{\perp \perp}$.
If $(q_n)$ is increasing then the last condition means that
 $\psi(q_n) \to 0$ for any $\psi \in A^{\perp}$.
\end{corollary}

\begin{proof}  By Theorem \ref{psvn}, the information about $q$ 
in the lines above the present corollary is equivalent to:
there is a (decreasing, if desired) sequence $(q_n)$ of
projections in $M$ with infimum $q$ in  $M^{**}$ lying in $A^{\perp \perp}$
satisfying $\varphi(q) = 1$, and $\psi(q) = 0$ for all
normal states $\psi$ of $M$.   As in \cite[Lemma 9.1]{BW}
the last condition
is equivalent to the infimum in $M$ of $(q_n)$ being $0$,
and  $\varphi(q) = 1$
if and only if $\varphi(q_n) = 0$ for all $n$.
Finally set $p = q^\perp$ and  replace $q_n$ by $q_n^\perp$.
\end{proof}

We remark that if $A$ is an Ueda algebra then it is easy to see that so
is  $A^* = \{ x^* : x \in A \}$.

\begin{corollary} \label{co1}     Suppose that a weak* closed
subalgebra $A$ of a von Neumann algebra $M$ is
an Ueda algebra.   If $\varphi \in M^*$ has  nonzero singular part $\varphi_s$, then  there exists a contraction $a \in A$ and a projection 
$p \in M^{**}$ with  $a^n \to p$ weak* in $M^{**}$,  $a^n \to 0$ weak* in $M$,
 and $\varphi_s = \varphi \cdot p$.
\end{corollary}

\begin{theorem} \label{co3}    Suppose that 
a weak* closed
subalgebra $A$ of a von Neumann algebra $M$ is
an Ueda algebra. Write $A^*_s$ and $A^*_n$ for the set of restrictions to $A$ of singular and normal functionals on $M$.   Each
$\varphi \in A^*$ has a unique Lebesgue decomposition relative to $M$: $\varphi = \varphi_n + \varphi_s$
with $\varphi_n \in A^*_n$ and $\varphi_s \in A^*_s$.   Moreover, $\Vert \varphi \Vert 
= \Vert \varphi_n \Vert  + \Vert  \varphi_s \Vert$.   
\end{theorem}

\begin{corollary} \label{co4}   Suppose that 
a weak* closed
subalgebra $A$ of a von Neumann algebra $M$ is
an Ueda algebra. Then the predual $A_*$ of $A$  
is unique, and is
an $L$-summand in $A^*$. 
  Also, $A_*$ has property {\rm  (V$^*$)} and is weakly sequentially complete. 
\end{corollary}

(See also e.g.\ \cite{KY} for recent similar results for a completely different class of dual operator algebras.)

\begin{theorem}[F. $\&$ M. Riesz type theorem] \label{co5}  Suppose that a weak* closed
subalgebra $A$ of a von Neumann algebra $M$ is
an Ueda algebra. 
If $\varphi \in M^*$ annihilates $A$ (that is, $\varphi \in A^\perp$) 
then the normal and singular parts, $\varphi_n$ and $\varphi_s$, also annihilate $A$.
\end{theorem}

Our proofs from \cite{BL3} then give the following results (suitably modified by an appeal to Theorem \ref{co5}
instead of to the F $\&$ M type theorem in \cite{BL3}), as noted in  \cite{U} and suggested by the referee of that paper.  
One may define an {\em F $\&$ M Riesz algebra}
to be a weak* closed subalgebra $A$ of a von Neumann algebra $M$,
such that if $\varphi \in A^\perp$  
then the normal and singular parts, $\varphi_n$ and $\varphi_s$, also annihilate $A$.   Theorem \ref{co5} then 
says that any Ueda algebra is an F $\&$ M Riesz algebra. 
Again, it is easy to argue (by considering $\psi^*(x) = \overline{\psi(x^*)}$)
that if  $A$  is an  F $\&$ M Riesz algebra then so is  $A^* = \{ x^* : x \in A \}$.
By proofs in  \cite{BL3}  we then have:

\begin{corollary} \label{co6}   Suppose that $A$ is an  F $\&$ M Riesz algebra in a von Neumann algebra
$M$ such that  $A + A^*$ is weak* dense in $M$.     If   $\varphi \in M^*$
annihilates $A + A^*$ then $\varphi$ is singular.   Any normal functional on $M$ is the unique Hahn-Banach extension 
of its restriction to $A+ A^*$, and  in particular is normed by $A+ A^*$.    In addition, any Hahn-Banach extension to $M$ of a weak* continuous functional on $A$, is normal.   
\end{corollary}

\begin{corollary} \label{co7}    If $A$ is  an  F $\&$ M Riesz algebra in a von Neumann algebra
$M$ such that  $A + A^*$ is weak* dense in $M$ then  ${\rm Ball}(A + A^*)$ is weak* dense in ${\rm Ball}(M)$.
\end{corollary}

Moreover in this case we obtain all the 
assertions of Theorem \ref{kap1} too.  

The last assertion of the Corollary \ref{co6}   is related to the well known Gleason-Whitney theorem in function theory.  A special case of the following appears in \cite[Theorem 4.1]{BL3} and \cite[Theorem 3.4]{Lis}.   We can express some of the ideas in those results more
abstractly and generally as follows:

\begin{lemma} \label{gw2}  Suppose $A$ is a weak* closed subalgebra of a von Neumann algebra $M$.  Then $A + A^*$ is weak* dense in $M$ 
if and only if there is at most one normal Hahn-Banach extension to $M$ of any normal weak* continuous functional on $A$.
\end{lemma}  

\begin{proof} 
  ($\Rightarrow$)  \  Choose $a \in {\rm Ball}(A)$ such that $\varphi(a) = 1$, and let  $e$ be the left support of $a$, which is the support of $aa^*$.  We may suppose that $\varphi \in M_*$
and that $\psi$ is another normal Hahn-Banach extension
of $\varphi_{|A}$.   We have $$1 = \varphi(a) \leq
|\varphi|(aa^*) \leq |\varphi|(e),$$
so that $|\varphi|(e^\perp) = 0$.  Hence
$|\varphi| e^\perp = 0$ and $\varphi e^\perp = 0$.
Similarly $\psi e^\perp = 0$ and $(\varphi - \psi)e^\perp = 0$.   
Next note that $\varphi a$ is contractive and unital,
so positive.   Similarly for  $\psi$, and so 
$(\varphi - \psi) a$ is a selfadjoint normal functional.  It vanishes on $A$, hence also on $A+A^*$
and on $M$.   From this it is easy to see that 
$(\varphi - \psi) e = 0$.    So $\varphi - \psi = 
(\varphi - \psi) e + (\varphi - \psi) e^\perp =  0$.

($\Leftarrow$)  \ It is enough to show that if normal $\psi$ annihilates $A+A^*$ then $\psi = 0$.
By taking real and imaginary parts we may assume that $\psi = \psi^*$.   Suppose $\psi = \psi_1 - \psi_2$
for positive normal $\psi_k$.   Then $\psi_1 = \psi_2 + \psi$ and $\psi_2$ agree on $A$, and are normal Hahn-Banach extensions since the norm of a positive functional is its value at $1$.   So $\psi_1 = \psi_2$ and $\psi = 0$.
\end{proof}

\begin{corollary}[Gleason-Whitney type theorem]   Suppose that $A$ is an  F $\&$ M Riesz algebra in a von Neumann algebra
$M$.   Then $A+A^*$ is weak* dense in $M$ 
if and only if every normal functional on $A$ has a 
unique  Hahn-Banach extension to $M$,  
and if and only if every normal
functional on $A$ has a
unique normal Hahn-Banach extension to $M$. \end{corollary}

Of course all of these hold when $A$ is a maximal subdiagonal subalgebra of 
a $\sigma$-finite von Neumann algebra $M$.    Conversely these properties characterize 
maximal subdiagonal subalgebras.  The following is a partial strengthening of \cite[Theorem 3.4]{Lis}
(the equivalence of (i) and (iv) there).   

\begin{corollary}[Gleason-Whitney type theorem] Let $A$ be a weak* closed unital subalgebra of a $\sigma$-finite von Neumann algebra $M$, for which
\begin{itemize} 
\item $\sigma^\nu_t(A)=A$ for each $t\in \mathbb{R}$ (where $\sigma^\nu_t$ is the modular 
automorphism group for $M$ described in our Introduction), and
 \item the canonical expectation $\mathcal{E}:M\to A\cap A^*=D$ is
 multiplicative on $A$.
\end{itemize} 
Then $A+A^*$ is weak* dense in $M$  (that is, $A$ is maximal subdiagonal with respect to $D$)  
if and only if every normal functional on $A$ has a 
unique normal Hahn-Banach extension to $M$.
 \end{corollary}

\section{The case of semi-finite and general von Neumann algebras}

We first briefly discuss the results of our paper in the setting of 
general von Neumann algebras.   We recall from e.g.\
\cite[Proposition 2.2.5]{Sakai} that
any  von Neumann algebra $M$ 
is a  direct sum of algebras $M_i$ of the form 
$R_i \bar{\otimes} B(\mathfrak{H}_i)$ for a $\sigma$-finite  von Neumann algebra  $R_i$.    
If $A$ is a maximal subdiagonal algebra
in $M$, and if the center of $M$ is contained in the center
of $A \cap A^*$, then the central projections 
corresponding to the direct sum will allow a decomposition of a maximal 
subdiagonal algebra $A$ of $M$ as a direct sum of algebras $A_i\subset M_i$, and it is easy to see that 
these are maximal 
subdiagonal subalgebras of $M_i$.   
Assuming that the $B(\mathfrak{H}_i)$'s appearing 
in the form of $M_i$ above correspond to separable Hilbert spaces $\mathfrak{H}_i$,
then  $R_i \bar{\otimes} B(\mathfrak{H}_i)$
is $\sigma$-finite, 
and  we get Ueda's theorem in this case (Theorem \ref{uadasig} but with the 
$\sigma$-finite  $M$ replaced by our $M$ above).   
  We immediately deduce that all
the 
results in the 
last section (Section \ref{cons}) are valid for this $A$ and $M$.

\medskip

We now turn to investigating when
 Ueda's peak set theorem fails.
 Of course if Ueda's peak set  theorem fails for a von Neumann algebra $M$ then it
also fails for every weak* closed unital subalgebra $A$ of  $M$.
Thus henceforth in this section we shall assume that  $A = M$.

An {\em Ulam measurable cardinal} is one such that if $I$ is  a set  of this cardinality, then there 
exists a 
free ultrafilter $p$ on $I$ such that every 
sequence $A_1 \supseteq A_2 \supseteq \cdots$ of nonempty sets in $p$ has nonempty intersection
\cite{CN,Jech}.    
(Remark: it is a pleasant exercise that an ultrafilter allows no empty countable intersections of members
if and only if it is closed under countable intersections.  Also, one may always make countable intersections `decreasing'.)  
The concept of
{\em measurable cardinal} used in the next result will be explained 
a little more at the start of its proof.    This result shows that there is
little chance of generalizing Ueda's peak set theorem to semi-finite von Neumann algebras in 
the usual set theoretic universe used in most of functional analysis, since 
this would imply a solution to one of the famous open ``problems''
in mathematics.    We use quotes because nowadays this
problem is not believed to be solvable.

  The strategy of our proof is simple: 
it is known that a bound on the size of a set $I$ in relation to being of measurable cardinality is equivalent to  being
`realcompact'.  Also, $I$ not being realcompact is known to imply that $\beta I \setminus I$  contains points not contained in closed G$_\delta$ subsets of $\beta I$ of a certain type.  Finally, for $C(K)$ spaces the closed G$_\delta$ sets are exactly the peak sets, by the strict form
of the Urysohn lemma or
as in Proposition \ref{cipeak3}.  
 However since we lack a good reference (besides scattered pieces
found in an internet search for `realcompact discrete spaces'; see e.g.\ \cite[p.\ 402 ff]{CN}), we will include short arguments for several of these points for the reader's convenience.  
   
\begin{theorem} \label{Ulam}   If Ueda's peak set theorem held for all finite von Neumann algebras then 
there exist no (uncountable) measurable cardinals.
\end{theorem}

\begin{proof}    The existence of (uncountable) measurable cardinals is known to be 
equivalent to the existence of Ulam measurable cardinals \cite{CN,Jech}.
Suppose that $I$ was an uncountable set of 
Ulam measurable cardinality.  
Clearly $M= \ell^\infty(I)$ is a finite (and semi-finite) von Neumann algebra, and $A = M$ is a maximal subdiagonal
algebra.   We may view the free ultrafilter $p$ in the definition of
Ulam measurable cardinality as a (singleton) closed set in $\beta I \setminus I$.  It is the support of a Dirac probability measure, 
which can be viewed as a pure state on $M = C(\beta I)$ (evaluation at $p$).    This state is singular (we leave this
as an exercise since there are many ways to see this).    Moreover via the well known correspondences between
minimal projections and pure states and their supports, the support of this
state in $C(\beta I)^{**}$    is  the  minimal projection which is the image $e$  of the 
characteristic function of $F = \{ p \}$ in $C(\beta I)^{**}$ (that is, it is the image of the functional $\mu \mapsto \mu (F)$ on $C(\beta I)^*$, viewing the latter as a space of measures).     Indeed here we are just invoking 
aspects of the well known noncommutative dictionary between the basic theory of
probability measures and that of states.

If Ueda's theorem held for $M$ then there would exist
a peak projection $q \in C(\beta I)^{**}$ with $e \leq q \leq z$, where $z$ is the orthogonal complement of  the canonical projection  in $M^{**}$ corresponding to $M_*$.    These three projections $e, z, q$ correspond to closed sets in $\beta I$, namely  
to sets $F  = \{ p \}, \beta I \setminus I,$ 
and $E$ say, respectively; and the latter is a classical peak set by the `peaking'
 theory \cite{BHN,Hay,Bnpi,BRead,BN,BReadII}.  (That $z$ corresponds to 
$\beta I \setminus I$ is well known, and was sketched in an earlier
version of the present paper available on arXiV.)

By Theorem \ref{psvn}, the characteristic function 
of any peak set $E$ for $M$ is 
an intersection of a decreasing sequence of projections in $M = C(\beta I) =
l^\infty(I)$.  Thus by the theory of the
Stone-Cech compactification, $E = \cap_{n=1}^\infty \, [A_n]$,
where $[A_n]$ is the (clopen) closure in $\beta I$ of  (open)  $A_n \subset I$,
where $A_1 \supseteq A_2 \supseteq \cdots$.
Also, $\cap_{n=1}^\infty \, [A_n] \subset \beta I \setminus I$
if and only if $\cap_{n=1}^\infty \, A_n = \emptyset$.  To see the latter,
note that $$\cap_{n=1}^\infty \, A_n \subset
(\cap_{n=1}^\infty \, [A_n]) \cap I \subset (\beta I \setminus I) \cap I = \emptyset$$ if
$\cap_{n=1}^\infty \, [A_n] \subset \beta I \setminus I$.
The converse follows from the inclusion
$$I \cap (\cap_{n=1}^\infty \, [A_n]) \subset I \cap [A_n] = A_n, \qquad
n \in \Ndb.$$    

Thus for any closed subset $F$
of a peak set  $E$ for $C(\beta I)$, with $E \subset \beta I \setminus I$,
we have $F \subset \cap_{n=1}^\infty \, [A_n]$  for sets $A_1 \supseteq A_2 \supseteq \cdots$ in $I$ 
with empty intersection.   In our special case where $F = \{ p \}$, the fact that $p \in [A_n]$ 
implies that $A_n \in p$ for all $n \in \Ndb$, with $p$ regarded as an ultrafilter.  This contradicts
the property  $p$ has in the definition of
Ulam measurable cardinality above.   So there is no Ulam measurable cardinal.
\end{proof}

\begin{remark}    By the last  proof 
  Ueda's peak set theorem holding for $M = A = \ell^\infty(I)$, is equivalent to 
saying that every closed set $F$ in $\beta I \setminus I$ which is 
the support of a Borel probability measure, is contained in $\cap_{n=1}^\infty \, [A_n]$  for sets $A_1 \supseteq A_2 \supseteq \cdots$ in $I$ with empty intersection.    Closed sets  in $\beta I \setminus I$  have a nice characterization
in the basic literature of the Stone-Cech compactification. 
\end{remark} 

It turns out that Ueda's peak set  theorem
also fails when $M = A = B(H)$ with $H$ of dimension an Ulam measurable cardinal,  
 or a real valued measurable cardinal, 
as is discussed together with Nik Weaver in \cite{BW}.    Indeed in 
 that paper (which was written after the first distributed version of  
the present paper) it is shown that  if  $M$ is a von Neumann algebra then Ueda's peak set  theorem
fails when $M = A$ if and only if $M$ possesses a singular state $\varphi$ which 
is {\em  regular}, that is, $\varphi(\vee_n \, q_n) = 0$ for every sequence of projections $(q_n)$  in ${\rm Ker}(\varphi)$.
(See \cite{Ham} for other characterizations and facts about regular states;
hence Ueda's peak set  theorem is strongly tied to `quantum measure theory'
in the sense of that reference.)   This is also equivalent to saying that there is
a collection of mutually orthogonal projections in $M$ of cardinality $\geq$
a fixed cardinal $\kappa$, namely the first cardinal on which there is a
 `regular' singular finitely additive
probability measure.   (Here and below measures are assumed to be defined
on all subsets of the cardinal.)
The existence of regular singular states or such regular measures
is generally believed to be consistent with ZFC set theory.    Indeed as 
explained in \cite{BW} it is believed to be consistent with ZFC set theory
that the latter `first cardinal' is $\leq$ the cardinality of the real numbers. 
On the other hand, since any cardinal on which there is a
 singular probability measure dominates 
the `first cardinal' above,
it follows that if $M \subset B(H)$ where dim$(H)$ is smaller than any 
 real-valued measurable cardinal   
(or if measurable cardinals do not exist), then Ueda's 
peak set theorem holds for $M$ (and taking $A = M$).

From the assertion in the last paragraph about the cardinality of the real numbers,
it follows that one should not hope to be able to prove 
 Ueda's
peak set theorem for $A = M = l^\infty(\Rdb)$ in ZFC.
Indeed Ueda's theorem in this case implies by
the assertion in the last paragraph
about regular states, a negative solution to the famous `Banach measure problem':
Is there a probability 
measure defined on all subsets of 
$[0,1]$ which is zero on singletons?  (It is well known that if there is, then one 
can find another that extends Lebesgue measure.) Banach showed that you cannot prove an affirmitive answer to this in ZFC.  
The existence of a negative answer 
is equivalent to the nonexistence of  measurable cardinals in ZFC.
  However as we have stated earlier,
it is generally believed by set theorists 
that the existence of measurable cardinals  is consistent with ZFC. 
 
This shows that one cannot hope to be able to prove Ueda's
peak set theorem in ZFC for von Neumann algebras that are 
much `bigger' than $\sigma$-finite (the case of the main theorem
of our paper).   And indeed experts in von Neumann algebras
are usually happy to only consider $\sigma$-finite von Neumann algebras in their results,
because `bigger' algebras are often pathological.  
 On the other
hand it is shown in \cite{BW} that Ueda's
peak set theorem holds in ZFC for $A = M = l^\infty(\aleph_1)$,
where $\aleph_1$ is the first uncountable cardinal,
and this von Neumann algebra is not $\sigma$-finite.
Hence if we assume the continuum hypothesis then 
Ueda's peak set theorem does hold for $A = M = l^\infty(\Rdb)$.
Assuming the negation of the continuum hypothesis, a remaining 
question seems to be for what cardinals $\kappa$ between  the infinite countable
cardinal 
and the cardinality of the reals can one prove Ueda's
peak set theorem in ZFC for 
$A = M = l^\infty(\kappa)$.   Nik Weaver has sketched to us a
proof in the case of $\aleph_2$, and this trick seems
to extend to $\aleph_n$ for $n \in \Ndb$. 

Thinking about the last paragraph in conjunction with the 
proof of our main theorem, suggests to us 
that it may possibly be interesting to study Haagerup's reduction theory,
the standard form, and related topics, for von Neumann algebras
possessing uncountable collections of mutually orthogonal projections
 of cardinality smaller than the cardinality of the reals
(assuming of course the negation of the continuum hypothesis).   

\medskip

{\em Acknowledgments.}  We are grateful to Nik Weaver for a very helpful and 
lengthy conversation  
shortly before  distribution of 
an earlier version of
the present paper, which led to \cite{BW}, a paper largely devoted to quantum measure theory 
in the sense of \cite{Ham}, quantum cardinals, and various
continuity properties of states on von Neumann algebras.
In Section 9 of that paper the
set theoretic issues associated with Ueda's peak set theorem 
for von Neumann algebras are explored more fully,
as mentioned above.       We are also grateful to the referee for many helpful comments.

\end{document}